\documentclass[12pt]{amsart}
\usepackage{amsfonts,amssymb,amscd,amsmath,enumerate,verbatim}
\usepackage[latin1]{inputenc}
\usepackage{amscd}
\usepackage{latexsym}
\usepackage{mathptmx}
\usepackage{multicol}
\input xy
\xyoption{all}
\usepackage{hyperref}
\usepackage{color}
\def\NZQ{\mathbb}               

\def\ZZ{{\NZQ Z}}
\def\RR{{\NZQ R}}

\def\frk{\mathfrak}               

\def\Phi{{\frk N}}
\def\wb{{\bold w}}

\def\opn#1#2{\def#1{\operatorname{#2}}} 
\opn\chara{char} 
\opn\length{\ell} 
\opn\pd{pd} 
\opn\rk{rk}
\opn\projdim{proj\,dim} 
\opn\injdim{inj\,dim} 
\opn\rank{rank}
\opn\depth{depth} 
\opn\grade{grade} 
\opn\height{height}
\opn\embdim{emb\,dim} 
\opn\codim{codim}

\opn\Tr{Tr} 
\opn\bigrank{big\,rank}
\opn\superheight{superheight}
\opn\lcm{lcm}
\opn\trdeg{tr\,deg}
\opn\reg{reg} 
\opn\lreg{lreg} 
\opn\ini{in} 
\opn\lpd{lpd}
\opn\size{size}
\opn\mult{mult}
\opn\dist{dist}
\opn\cone{cone}
\opn\lex{lex}
\opn\rev{rev}
\opn\bipyr{bipyr}
\opn\div{div} \opn\Div{Div} \opn\cl{cl} \opn\Cl{Cl}
\opn\Spec{Spec} \opn\Supp{Supp} \opn\supp{supp} \opn\Sing{Sing}
\opn\Ass{Ass} \opn\Min{Min}
\opn\Ann{Ann} \opn\Rad{Rad} \opn\Soc{Soc}
\opn\Syz{Syz} \opn\Im{Im} \opn\Ker{Ker} \opn\Coker{Coker}
\opn\Am{Am} \opn\Hom{Hom} \opn\Tor{Tor} \opn\Ext{Ext}
\opn\End{End} \opn\Aut{Aut} \opn\id{id} \opn\ini{in}

\opn\nat{nat}
\opn\pff{pf}
\opn\Pf{Pf} \opn\GL{GL} \opn\SL{SL} \opn\mod{mod} \opn\ord{ord}
\opn\Gin{Gin}
\opn\Hilb{Hilb}\opn\adeg{adeg}\opn\std{std}\opn\ip{infpt}
\opn\Pol{Pol}
\opn\sat{sat}
\opn\Var{Var}
\opn\Gen{Gen}

\opn\aff{aff} \opn\con{conv} \opn\relint{relint} \opn\st{st}
\opn\lk{lk} \opn\cn{cn} \opn\core{core} \opn\vol{vol}
\opn\link{link} \opn\star{star}
\opn\gr{gr}

\def\Pc{{\mathcal P}}

\def\Rc{{\mathcal R}}

\def\vol{{\textnormal{vol}}}
\def\conv{{\textnormal{conv}}}

\def\ord{{\textnormal{ord}}}
\def\pot#1#2{#1[\kern-0.28ex[#2]\kern-0.28ex]}

\opn\dirlim{\underrightarrow{\lim}}
\opn\inivlim{\underleftarrow{\lim}}
\def\Implies{\ifmmode\Longrightarrow \else
	\unskip${}\Longrightarrow{}$\ignorespaces\fi}
\def\implies{\ifmmode\Rightarrow \else
	\unskip${}\Rightarrow{}$\ignorespaces\fi}
\def\iff{\ifmmode\Longleftrightarrow \else
	\unskip${}\Longleftrightarrow{}$\ignorespaces\fi}

\let\:=\colon
\newtheorem{Theorem}{Theorem}[section]

\newtheorem{Remark}[Theorem]{Remark}

\newtheorem{Example}[Theorem]{Example}

\newtheorem{Question}[Theorem]{Question}
\numberwithin{equation}{section}

\let\epsilon\varepsilon
\let\phi=\varphi
\let\kappa=\varkappa
\def\qed{\ifhmode\textqed\fi
	\ifmmode\ifinner\quad\qedsymbol\else\dispqed\fi\fi}
\def\textqed{\unskip\nobreak\penalty50
	\hskip2em\hbox{}\nobreak\hfil\qedsymbol
	\parfillskip=0pt \finalhyphendemerits=0}
\def\dispqed{\rlap{\qquad\qedsymbol}}

\opn\dis{dis}
\opn\height{height}
\opn\dist{dist}
\def\pnt{{\raise0.5mm\hbox{\large\bf.}}}

\opn\Lex{Lex}
\opn\conv{conv}
\opn\Ehr{Ehr}

\begin{document}
\title{Depth of an initial ideal}
	\author[T.~Hibi]{Takayuki Hibi}
\address[Takayuki Hibi]{Department of Pure and Applied Mathematics,
	Graduate School of Information Science and Technology,
	Osaka University,
	Suita, Osaka 565-0871, Japan}
\email{hibi@math.sci.osaka-u.ac.jp}
\author{Akiyoshi Tsuchiya}
\address[Akiyoshi Tsuchiya]
{Graduate school of Mathematical Sciences,
University of Tokyo,
Komaba, Meguro-ku, Tokyo 153-8914, Japan} 
\email{akiyoshi@ms.u-tokyo.ac.jp}

\thanks{The first author was partially supported by JSPS KAKENHI 19H00637.  The second author was partially supported by JSPS KAKENHI 19K14505.}
\begin{abstract}
Given an arbitrary integer $d>0$, we construct a homogeneous ideal $I$ of the polynomial ring $S = K[x_1, \ldots, x_{3d}]$ in $3d$ variables over a filed $K$ for which $S/I$ is a Cohen--Macaulay ring of dimension $d$ with the property that, for each of the integers $0 \leq r \leq d$, there exists a monomial order $<_r$ on $S$ with ${\rm depth} (S/{\rm in}_{<_r}(I)) = r$, where ${\rm in}_{<_r}(I)$ is the initial ideal of $I$ with respect to $<_r$. 
\end{abstract}
\maketitle
\thispagestyle{empty}
\section{Background}
In order to answer a question suggested in \cite[p.~38]{DEP}, the first author \cite[p.~285]{Hibi86} discovered a graded Gorenstein Hodge algebra  whose corresponding discrete Hodge algebra is not a Cohen--Macaulay ring.  In the modern language of Gr\"obner bases and initial ideals, the work guarantees the existence of a homogeneous ideal $I$ of the polynomial ring $S = K[x_1, \ldots, x_n]$ over a field $K$ for which $S/I$ is Cohen--Macaulay with the property that there is an initial ideal ${\rm in}_{<}(I)$ of $I$ for which $S/({\rm in}_{<}(I))$ is not Cohen--Macaulay.  On the other hand, in \cite[Corollary 3.9]{CV}, it is shown that if $A$ is an ASL (algebra with straightening laws \cite{Eis}) and $A_0$ is its discrete ASL, then $\depth A = \depth A_0$.  In particular the discrete ASL of a Cohen--Macaulay ASL is again Cohen--Macaulay. 

Take the above background into consideration, one cannot escape the temptation to study the question as follows:

\begin{Question}
\label{question}
{\em
Given an arbitrary integer $d>0$, does there exist a homogeneous ideal $I$ of the polynomial ring $S$ over a filed $K$ for which $S/I$ is a Cohen--Macaulay ring of dimension $d$ with the property that, for each of the integers $0 \leq r \leq d$, there is a monomial order $<_r$ on $S$ with ${\rm depth} (S/{\rm in}_{<_r}(I)) = r$, where ${\rm in}_{<_r}(I)$ is the initial ideal of $I$ with respect to $<_r$?
}
\end{Question}

The purpose of the present paper is to solve Question \ref{question} and, in addition, to supply related questions.

\section{Result}
We refer the reader to \cite[Chapter 2]{hhGTM} for fundamental materials and standard notation on Gr\"obner bases.  

Let $S=K[x_1,\ldots,x_n]$ denote the polynomial ring in $n$ variables over a field $K$.  Given a vector $\wb=(w_1,\ldots,w_d) \in \ZZ^n$ and a monomial $u=x_1^{a_1}\cdots x_n^{a_n} \in S$, the \textit{weight} of $u$ with respect to $\wb$ is defined to be $a_1w_1+\cdots+a_nw_n$.

\begin{Theorem}
\label{depth}
Given an arbitrary integer $d>0$, there exists a homogeneous ideal $I$ of the polynomial ring $S = K[x_1, \ldots, x_{3d}]$ in $3d$ variables over a filed $K$ for which $S/I$ is a Cohen--Macaulay ring of dimension $d$ with the property that, for each of the integers $0 \leq r \leq d$, there is a monomial order $<_r$ on $S$ with ${\rm depth} (S/{\rm in}_{<_r}(I)) = r$.
\end{Theorem}
\begin{proof}
	{\bf (First Step)} \, Let $d=1$.
	Let $S=K[x_1,x_2,x_3]$ and 
	\[
	I=(x_1^2-x_2x_3, x_1x_2-x_3^2,x_1x_3-x_2^2) \subset S
	\]
	(\cite[Example 3.3.6]{hhGTM}).
	Then $S/I$ is a one-dimensional Cohen--Macaulay ring.
	Let $\prec$ be the lexicographic order on $S$ with
	$x_3 \prec x_2 \prec x_1$.  Let, in addition, $\wb_0=(1,1,1)$ and $\wb_1=(1,2,2)$.
	For each $r \in \{0,1\}$, we introduce the monomial order $<_r$ on $S$ as follows:  One has $u <_r v$ if and only if one of the following holds:
		\begin{itemize}
		\item The weight of $u$ is less than that of $v$ with respect to $\wb_r$;
		\item The weight of $u$ is equal to that of $v$ with respect to $\wb_r$ and $u \prec v$.
		\end{itemize}
Then 
\[
\{x_1^2-x_2x_3, x_1x_2-x_3^2, x_1x_3-x_2^2, x_2^3-x_3^3\}
\]
is a Gr\"{o}bner basis of $I$ with respect to $<_0$ and
	${\rm depth} (S/{\rm in}_{<_0}(I)) = 0$.
	On the other hand, 
	\[
	\{x_2x_3-x_1^2, x_3^2-x_1x_2, x_2^2-x_1x_3\}
	\]
	is a Gr\"{o}bner basis of $I$ with respect to $<_1$ and ${\rm depth} (S/{\rm in}_{<_1}(I)) = 1$.
	
	\medskip
		{\bf (Second Step)}  \, Let $d>1$.
	Let $S=K[x_1,\ldots,x_{3d}]$ and 
	\[
I = (\{ \, x_{3i-2}^2-x_{3i-1}x_{3i}, \, \, x_{3i-2}x_{3i-1}-x_{3i}^2, \, \, x_{3i-2}x_{3i}-x_{3i-1}^2 \, : \, 1 \leq i \leq d \, \}) \subset S.
	\]
	Let $S_i=K[x_{3_i-2}, x_{3i-1}, x_{3i}]$ and 
	\[
	I_i = (x_{3i-2}^2-x_{3i-1}x_{3i}, x_{3i-2}x_{3i-1}-x_{3i}^2,x_{3i-2}x_{3i}-x_{3i-1}^2) \subset S_i,
    \]
	where $1 \leq i \leq d$.  Thus
	\[
S/I \cong S_1/I_1 \otimes_{K} \cdots \otimes_{K} S_d/I_d 
\]
and $S/I$ is a Cohen--Macaulay ring of dimension $d$.

Now, we employ the lexicographic order $\prec$ on $S$ with
\[
x_{3d} \prec x_{3d-1} \prec \cdots \prec x_1
\] 
on $S$ and the vectors
\[
\wb_r=(\underbrace{1,2,2, \ldots, 1,2,2}_{3r}, \underbrace{1,1,1,\ldots, 1,1,1}_{3(d-r)}) \in \ZZ^{3d}, \, \, \, \, \, 0 \leq r \leq d.
\]
	For each $0 \leq r \leq d$, we introduce the monomial order $<_r$ on $S$ as follows:  One has $u <_r v$ if and only if one of the following holds:
		\begin{itemize}
		\item The weight of $u$ is less than that of $v$ with respect to $\wb_r$;
		\item The weight of $u$ is equal to that of $v$ with respect to $\wb_r$ and $u \prec v$.
		\end{itemize}
It then follows that the set $A \cup B$, where
	\[
	A = \{ \, x_{3i-1} x_{3i} -x_{3i-2}^2, x_{3i-1}^2-x_{3i-2} x_{3i}, x_{3i}^2-x_{3i-2} x_{3i-1} \, : \, 1 \leq i \leq r \, \} 
	\] and \[
	B = \{ \, x_{3i-2}^2-x_{3i-1}x_{3i}, x_{3i-2}x_{3i-1}-x_{3i}^2, x_{3i-1}x_{3i}-x_{3i}^2, x_{3i-1}^3-x_{3i}^3 \, : \, r+1 \leq i \leq d \, \}, 
	\]
is a Gr\"{o}bner basis of $I$ with respect to $<_r$.
Since 
	\begin{eqnarray*}
	S/{\rm in}_{<_r}(I) & \cong &
\dfrac{K[x_1, x_2, x_3]}{(x_2 x_3, x_3^2, x_2^2)} \otimes_K \cdots \otimes_K \dfrac{K[x_{3r-1}, x_{3r-1}, x_{3r}]}{(x_{3r-1} x_{3r}, x_{3r}^2, x_{3r-1}^2)} \\ &&\otimes_K \dfrac{K[x_{3r+1}, x_{3r+2}, x_{3r+3}]}{(x_{3r+1}^2, x_{3r+1} x_{3r+2}, x_{3r+1} x_{3r+3}, x_{3r+2}^3)} \\ && \otimes_K \cdots \\
&& \otimes_K \dfrac{K[x_{3d-2}, x_{3d-1}, x_{3d}]}{(x_{3d-2}^2, x_{3d-2} x_{3d-1}, x_{3d-2} x_{3d}, x_{3d-1}^3)},
	\end{eqnarray*}
	one has ${\rm depth} (S/{\rm in}_{<_r}(I)) = r$, as desired.
\end{proof}

\begin{Remark}
{\em Let $S/I$ be the quotient ring studied in the proof of Theorem \ref{depth} and $\reg(S/I)$ its regularity.  On has $\reg(S/I) = d$.  Furthermore, it follows that
\[
{\rm depth} (S/{\rm in}_{<_r}(I)) + \reg(S/{\rm in}_{<_r}(I)) = 2d
\]
for each of $0 \leq r \leq d$. 
}
\end{Remark}

\section{Questions}
We conclude the present paper with related questions.

\medskip

\begin{Question}
{\em In Question \ref{question}, one may ask if $S/I$ can be a Gorenstein ring. 
}
\end{Question}

\begin{Question}
{\em In Question \ref{question}, one may ask if $I$ can be a prime ideal. 
}
\end{Question}

\begin{Question}
{\em Let $I \subset S = K[x_1, \ldots, x_n]$ be a homogeneous ideal with $\reg(S/I) = r$ and ${\rm depth} (S/I) = e$.  Suppose that there is a monomial order $\prec$ on $S$ with $\reg(S/{\rm in}_{\prec}(I)) = r'$ and ${\rm depth} (S/{\rm in}_{\prec}(I)) = e'$.  Then, for each $r \leq r'' \leq r'$ and for each $e' \leq e'' \leq e$, does there exist a monomial order $\prec'$ on $S$ with $\reg(S/{\rm in}_{\prec'}(I)) = r''$ and ${\rm depth} (S/{\rm in}_{\prec'}(I)) = e''$? 
}
\end{Question}

Let $I \subset S = K[x_1, \ldots, x_n]$ be the toric ideal of a {\em unimodular} convex polytope \cite[p.~107]{HHO}.  Thus, for {\em any} monomial order $\prec$ on $S$, its initial ideal ${\rm in}_{\prec}(I)$ is generated by squarefree monomials.  It then follows from \cite{CV} that $\reg(S/{\rm in}_{\prec}(I)) = \reg(S/I)$ and ${\rm depth} (S/I) = {\rm depth} (S/{\rm in}_{\prec}(I))$.

\begin{Question}
{\em Is there a nice class of toric ideals $I \subset S$ for which $\reg(S/{\rm in}_{\prec}(I)) = \reg(S/I)$ for {\em any} monomial order $\prec$ on $S$ and for which there is a monomial order $\prec$ on $S$ with  ${\rm depth} (S/{\rm in}_{\prec}(I)) < {\rm depth} (S/I)$?  
}
\end{Question}

Let $D$ be a finite distributive lattice and $\Rc_K[D]$ the ASL on $D$ over a field $K$ (\cite[p.~98]{Hibi87}).  It is known that $\Rc_K[D]$ is normal and Cohen--Macaulay.  Its discrete ASL is Cohen--Macaulay.  Let $S = K[\{x_\alpha : \alpha \in D\}]$ the polynomial ring in $|D|$ variables over $K$.  The defining ideal of $\Rc_K[D]$ is
\[
I_D = ( \, \{ \, x_\alpha x_\beta - x_{\alpha \wedge \beta} x_{\alpha \vee \beta} \, : \, \alpha \not\leq \beta, \, \beta \not\leq \alpha \, \} \, ) \subset S.
\] 

\begin{Question}
{\em For which finite distributive lattices $D$, does there exist a monomial order $\prec$ on $S = K[\{x_\alpha : \alpha \in D\}]$ with $\min{\rm depth} (S/{\rm in}_{\prec}(I_D)) = 0$?}
\end{Question}



\end{document}